\pdfoutput=1
\RequirePackage{ifpdf}
\ifpdf 
\documentclass[pdftex]{sigma}
\else
\documentclass{sigma}
\fi

\usepackage{mathrsfs}

\newtheorem{Theorem}{Theorem}[section]
\newtheorem{Corollary}[Theorem]{Corollary}

\newtheorem{Proposition}[Theorem]{Proposition}
 { \theoremstyle{definition}

 }

\numberwithin{equation}{section}

\newcommand{\diag}{\mathrm{diag}}

\newcommand{\Sym}{\mathrm{Sym}}
\newcommand{\Exp}{\mathrm{Exp}}

\def\O{\mathcal{O}}

\newcommand{\vir}{\mathrm{vir}}

\newcommand{\ch}{\mathrm{ch}}

\def\C{\mathbb C}
\def\Q{\mathbb Q}
\def\Z{\mathbb Z}

\def\P{\mathbb P}

\def\1{\mathbbm 1}

\begin{document}
\allowdisplaybreaks

\newcommand{\arXivNumber}{2201.07392}

\renewcommand{\thefootnote}{}

\renewcommand{\PaperNumber}{078}

\FirstPageHeading

\ShortArticleName{K-Theoretic Descendent Series for Hilbert Schemes of Points on Surfaces}

\ArticleName{K-Theoretic Descendent Series for Hilbert Schemes\\ of Points on Surfaces\footnote{This paper is a~contribution to the Special Issue on Enumerative and Gauge-Theoretic Invariants in honor of Lothar G\"ottsche on the occasion of his 60th birthday. The~full collection is available at \href{https://www.emis.de/journals/SIGMA/Gottsche.html}{https://www.emis.de/journals/SIGMA/Gottsche.html}}}

\Author{Noah ARBESFELD}

\AuthorNameForHeading{N.~Arbesfeld}

\Address{Department of Mathematics, Huxley Building, Imperial College London, London SW7 2AZ, UK}
\Email{\href{mailto:n.arbesfeld@imperial.ac.uk}{n.arbesfeld@imperial.ac.uk}}
\URLaddress{\url{math.columbia.edu/~nma/}}

\ArticleDates{Received January 28, 2022, in final form October 03, 2022; Published online October 16, 2022}

\Abstract{We study the holomorphic Euler characteristics of tautological sheaves on Hilbert schemes of points on surfaces. In particular, we establish the rationality of K-theoretic descendent series. Our approach is to control equivariant holomorphic Euler characteristics over the Hilbert scheme of points on the affine plane. To do so, we slightly modify a~Mac\-do\-nald polynomial identity of Mellit.}


\Classification{14C05; 14C17; 05E05}

\renewcommand{\thefootnote}{\arabic{footnote}}
\setcounter{footnote}{0}

\section{Introduction}
\subsection{K-theoretic descendent series}

Let $S$ be a nonsingular projective algebraic surface. For $n\geq 0$, let $S^{[n]}$ denote the Hilbert scheme of $n$ points on $S$.

The Hilbert schemes carry natural K-theory classes induced from classes on $S$. Let $\Sigma_n\subset S^{[n]}\times S$ denote the universal family and $\pi_{S^{[n]}}$ and $\pi_S$ denote the projections from $S^{[n]}\times S$ onto the corresponding factors. A class $\alpha\in K(S)$ induces a tautological class $\alpha^{[n]}\in K(S^{[n]})$ given~by \begin{gather}\label{tautdef}
\alpha^{[n]}={\pi_{S^{[n]}}}_*(\mathcal{O}_{\Sigma_n}\otimes \pi_S^*\alpha ).
\end{gather}

In this paper, we study the structure of holomorphic Euler characteristics of tautological classes. Namely, we consider the following {\em K-theoretic descendent series}: for classes $\alpha_1,\dots,\alpha_l \in K(S)$ and integers $k_1,\dots,k_l\geq 0$, set \begin{align}\label{projseries}
Z_{S}(\alpha_1,\dots,\alpha_l\mid k_1,\dots,k_l)=\sum_{n\geq 0} q^n \chi\big(S^{[n]}, \wedge^{k_1} \alpha^{[n]}_1 \otimes \dots \otimes \wedge^{k_l} \alpha_l^{[n]} \big)\in \Q[[q]].
\end{align}
Our first result is a positive answer to Question 5 of~\cite{AJLOP}.

\begin{Theorem}\label{descrat}
The series $Z_{S}(\alpha_1,\dots,\alpha_l\mid k_1,\dots,k_l)$ is the Laurent expansion of a rational function $F(q)/(1-q)^{\chi(\mathcal{O}_S)},$ where $F(q)$ is a polynomial of degree at most $k_1+\cdots+k_l$.
\end{Theorem}

In anticipation of applications to descendent series for Quot schemes
on simply-connected surfaces of geometric genus~0 (see~\cite[Section~3.4]{AJLOP}), we also record the following partial generalization in the case where one of the~$\alpha_i$ is the class of a line bundle.

\begin{Theorem}\label{lbdescrat}
If $\alpha_1$ is the class of a line bundle and $k_2,\dots,k_l\in \Z$ are fixed, then
\begin{align}\label{lbspec}
\sum_{k_1=0}^{\infty} (-m)^{k_1} Z_{S}(\alpha_1,\dots,\alpha_l\mid k_1,\dots,k_l)
\end{align}
is the Laurent expansion of a rational function $G(q,m)\cdot (1-qm)^r/(1-q)^{\chi(\mathcal{O}_S)},$ where $G(q,m)$ is a polynomial and $r\in \Z$.
\end{Theorem}

\subsection{Examples}
The simplest example
\begin{gather*}
Z_{S}(\varnothing\mid\varnothing)= \sum_{n\geq 0} q^n \chi\big(S^{[n]},\O_{S^{[n]}}\big)=\frac{1}{(1-q)^{\chi(\mathcal{O}_S)}}
\end{gather*}
 of Theorem~\ref{descrat} is computed in~\cite[Proposition 3.3(b)]{G}.
Let $\alpha\in K(S)$. The second simplest instance of Theorem~\ref{descrat} is the identity
\begin{gather*}
Z_{S}(\alpha\mid 1)=\sum_{n\geq 0} q^n \chi\big(S^{[n]},\alpha^{[n]}\big)=\frac{ \chi(\alpha) q}{(1-q)^{\chi(\mathcal{O}_S)}},
\end{gather*}
which is a consequence of~\cite[Corollary~1.3]{D} when $S$ is Fano and $\alpha$ is the class of an ample or trivial line bundle. The identity for arbitrary $S$ and $\alpha$ then follows from~\cite[Theorem~4.2]{EGL}.

The complexity of the numerators of the rational functions (\ref{projseries}) can grow as $l$, $k_i$ and the ranks of $\alpha_i$ increase. For example, if $\mathcal{L}\in K(S)$ is a line bundle, then by~\cite[Theorem~5.25]{Sc2},
\begin{align*}
Z_S(-\mathcal{L}\mid 3)&=-\sum_{n\geq 0} q^{n}\chi\big(S^{[n]},\Sym^3 \mathcal{L}^{[n]}\big)
\\
&= \frac{\big(\chi\big(\mathcal{L}^{\otimes 2}\big)\chi(\mathcal{L})-\chi\big(\mathcal{T}^*S\otimes \mathcal{L}^{\otimes 3}\big) \big)\big(q^3-q^2\big)+\chi\big(\mathcal{L}^{\otimes 3}\big) \big(q^2-q\big)-\binom{\chi(\mathcal{L})+2}{3} q^3}{(1-q)^{\chi(\mathcal{O}_S)}}.
\end{align*}
Other computations of examples of K-theoretic descendent series (\ref{projseries}) can be found in~\cite[Section~6]{A},~\cite[Corollary~8.11]{K},~\cite[Section~5]{Sc1},~\cite[Theorem~5.25]{Sc2} and~\cite[Section~7]{Z}.

This paper's approach can be used to produce new formulas for descendent series. For example, we compute the following in Section~\ref{compdesc}.
\begin{Proposition} \label{newdesc}
If $\mathcal{V}\in K(S)$ is a rank $3$ vector bundle, then
\begin{align*}
Z_{S}(\mathcal{V}|3) = \frac{1}{(1-q)^{\chi(\mathcal{O}_S)}} \bigg(&\big(\chi\big(\mathcal{T}^{*}S\otimes \wedge^3 \mathcal{V}\big)+\chi\big({\wedge}^2\mathcal{V}\otimes \mathcal{V}\big)-\chi\big({\wedge}^2\mathcal{V}\big)\chi( \mathcal{V})\big)\big(q^3-q^2\big)
\\
&+\chi\big({\wedge}^3\mathcal{V}\big)\big(q-q^3\big)+\binom{\chi(\mathcal{V})}{3} q^3\bigg).
\end{align*}
\end{Proposition}

We now present examples of the form (\ref{lbspec}). Let $\mathcal{L}$ be a line bundle on $S$. The simplest example
\begin{gather*}
\sum_{k=0}^{\infty} (-m)^kZ_{S}(\mathcal{L}\mid k)
=\sum_{k,n\geq 0} q^n(-m)^k\chi\big(S^{[n]},\wedge^{k}\mathcal{L}^{[n]}\big)
=\frac{(1-qm)^{\chi(\mathcal{L})}}{(1-q)^{\chi(\mathcal{O}_S)}}
\end{gather*}
of Theorem~\ref{lbdescrat} is a consequence of~\cite[Theorem~5.2.1]{Sc1}.
The next simplest example
{\samepage\begin{align}\label{descex}
\sum_{k=0}^{\infty} &(-m)^kZ_{S}(\mathcal{L},\alpha\mid k,1)= \sum_{n,k\geq 0} q^n(-m)^k \chi\big(S^{[n]},(\wedge^{k}\mathcal{L}^{[n]}) \otimes \alpha^{[n]}\big)\nonumber
\\
&=\frac{(1-qm)^{\chi(\mathcal{L})}}{(1-q)^{\mathcal{\chi(\mathcal{O}_S)}}} \sum_{n=0}^{\infty}q^{n+1}m^n\chi\big(\mathcal{L}^{\otimes n}\otimes \alpha\big)-q^{n+1}m^{n+1}\chi\big(\mathcal{L}^{\otimes (n+1)}\otimes \alpha\big)
\end{align}}%
is computed in~\cite[Proposition 20]{AJLOP}. Hirzebruch--Riemann--Roch implies that the series (\ref{descex}) is of the form predicted by Theorem~\ref{lbdescrat}.

\subsection{Comparison with cohomological descendents and other geometries}
\subsubsection{Cohomological descendent series}
The rationality of (\ref{projseries}) contrasts with the expected behavior of descendent series in cohomology. Cohomological descendent integrals are often packaged as follows: for classes $\alpha_1,\dots,\alpha_l\in K(S)$ and integers $k_1,\dots,k_l\geq 0$, form the series \begin{align}\label{cohseries}
\sum_{n\geq 0} q^n\int_{S^{[n]}} \ch_{k_1}\big(\alpha_1^{[n]}\big)\cdots \ch_{k_l}\big(\alpha_l^{[n]}\big) c\big(\mathcal{T}S^{[n]}\big).
\end{align}

The series (\ref{cohseries}) have a different flavor than their K-theoretic counterparts (\ref{projseries}). For example, by G\"{o}ttsche's formula~\cite[Theorem~0.1]{G},
\begin{gather*}
\sum_{n\geq 0} q^{n} \int_{S^{[n]}} c_{2n}\big(\mathcal{T}S^{[n]}\big)=\prod_{m>0} \frac{1}{(1-q^m)},
\end{gather*}
and by~\cite[Corollary 3]{CO},
\begin{gather*}
\sum_{n\geq 0}q^{n} \int_{S^{[n]}} \ch_{1}\big(\mathcal{O}^{[n]}\big)c_{2n-1}\big(\mathcal{T}S^{[n]}\big)=\frac{c_1(S)^2}{2} \frac{E_2(q)-E_3(q)}{\prod_{m>0}(1-q^m) };
\end{gather*}
here,
\begin{gather*}
E_k(q)=\sum_{n> 0} n^{k-1} \frac{q^n}{1-q^{n}}.
\end{gather*}

Instead,~\cite[Conjecture 2]{O} conjectures that a suitably normalized version of the series (\ref{cohseries}) belongs to a distinguished algebra of $q$-series called $q$-multiple zeta values. One result in this direction is~\cite[Theorem~2]{C}, in which a $\C^*$-equivariant version of (\ref{cohseries}) for the affine plane $\C^2$ is proved to be a quasimodular form.

\subsubsection{Curves}

Analogs of series (\ref{projseries}) and (\ref{cohseries}) can be studied for integrals over Hilbert schemes of points on curves. Given a nonsingular projective curve $C$, classes $\beta_1,\dots,\beta_l\in K(C)$ and integers $k_1,\dots,k_l\geq 0$, both the K-theoretic descendent series \begin{align}\label{curvekseries}
\sum_{n\geq 0} q^n \chi\big({C^{[n]}}, \wedge^{k_1}\beta_1^{[n]}\otimes \dots \otimes\wedge^{k_l}\beta_l^{[n]}\big)
\end{align}
and the cohomological descendent series
\begin{align}\label{curvecohseries}
\sum_{n\geq 0} q^n \int_{C^{[n]}} \ch_{k_1}\big(\beta_1^{[n]}\big)\cdots \ch_{k_l}\big(\beta_l^{[n]}\big) c\big(\mathcal{T}C^{[n]}\big)
\end{align}
are Laurent expansions of rational functions. This rationality follows from the methods of~\cite[Section~2.3]{OP}. Namely, the induction scheme of~\cite{EGL} (whose consequences we recall for Hilbert schemes on surfaces in Section~\ref{secuniv}) reduces the problem to cases where $C$ is $\P^1$. The series (\ref{curvekseries}) and (\ref{curvecohseries}) can be explicitly computed for this geometry using the relation
\begin{gather*}
\O(d)^{[n]}=(d+1)\O-(d+1-n)\O(-1)\in K\big(\big(\P^1\big)^{[n]}\big)\cong K(\P^n)
\end{gather*}
from the proof of~\cite[Theorem~2]{MOP2}.

\subsubsection{Quot schemes}
Tautological integrals over Quot schemes
parametrizing quotients (of dimension at most 1) of vector bundles on surfaces have been studied in~\cite{AJLOP,B2,B1,JOP,St1,St2}. Such Quot schemes are typically singular but carry perfect obstruction theories; see, for example,~\cite[Section~4]{St2}. Descendent series of Quot schemes can therefore be defined through virtual structures.

Hilbert schemes of points on surfaces, in particular, can be regarded as Quot schemes parametrizing finite length quotients
\begin{gather*}
\O_{S}\twoheadrightarrow \mathcal{Z}.
\end{gather*}
The associated virtual structure sheaves and virtual fundamental classes differ from the ordinary structure sheaves and fundamental classes of Hilbert schemes, and in fact give rise to more easily understood invariants. The virtual structures have explicit descriptions:
\begin{gather*}
\big[S^{[n]}\big]^{\vir}=(-1)^{n} e\big(\mathcal{K}_S^{[n]}\big)\cap \big[S^{[n]}\big],\qquad
\O^{\vir}_{S^{[n]}}=\sum_{k=0}^{n} (-1)^{k} \wedge^{k} \mathcal{K}_S^{[n]},\\ \mathcal{T}^{\vir}S^{[n]}=TS^{[n]}-\big(\mathcal{K}_S^{[n]}\big)^{*}.
\end{gather*}

Given classes $\alpha_1,\dots,\alpha_l\in K(S)$ and integers $k_1,\dots,k_l\geq 0,$ one can form the {\em virtual cohomological descendent series}
\begin{align}\label{vircohseries}
\sum_{n\geq 0} q^n\int_{[S^{[n]}]^{\vir}} \ch_{k_1}\big(\alpha_1^{[n]}\big)\cdots \ch_{k_l}\big(\alpha_l^{[n]}\big) c\big(\mathcal{T}^{\vir}S^{[n]}\big)
\end{align}
and the {\em virtual K-theoretic descendent series}
\begin{align}\label{virkseries}
\sum_{n \geq 0} q^n \chi\big(S^{[n]}, \wedge^{k_1} \alpha^{[n]}_1\otimes \dots\otimes \wedge^{k_l} \alpha^{[n]}_l\otimes \mathcal{O}^{\vir}\big).
\end{align}

In contrast to (\ref{cohseries}), the series (\ref{vircohseries}) is proved in~\cite[Theorem~2]{JOP} to be the Laurent expansion of a rational function. The series (\ref{virkseries}) is proved in~\cite[Theorem~1]{AJLOP} to be the Laurent expansion of a rational function; it is a consequence of~\cite[Theorem~4]{AJLOP} that this rational function can be written with denominator $(1-q)^{2(k_1+\dots+k_l)}$.

We remark that the rationality of (\ref{virkseries}) also follows from Theorem~\ref{lbdescrat}; in (\ref{lbspec}), one specializes $\alpha_1=\mathcal{K}_S$ and $m=1$. Unlike in (\ref{projseries}), the order of the pole of (\ref{virkseries}) at $q=1$ need not depend on~$S$. For example, by~\cite[Example 7]{AJLOP} or by substituting $m=1$ and $\mathcal{L}=\mathcal{K}_S$ in (\ref{descex}),
\begin{gather*}
\sum_{n\geq 0} q^n \chi\big(S^{[n]},\alpha^{[n]}\otimes \mathcal{O}^{\vir}\big)=-c_1(\alpha) \cdot \mathcal{K}_S\cdot \frac{q}{1-q}-\mathrm{rk}(\alpha)\cdot \mathcal{K}_S^2\cdot \frac{q^2}{(1-q)^2}.
\end{gather*}

The following observation accounts for the similarity in the behavior of descendent series for Hilbert schemes on curves and virtual descendent series for Hilbert schemes on surfaces: if $S$ is a smooth projective surface admitting a smooth canonical curve $C\in |K_S|$, then the virtual structures $\O^{\vir}_{S^{[n]}}$ and $\big[S^{[n]}\big]^{\vir}$ localize onto $C^{[n]}$.
Namely, if $\iota\colon C^{[n]}\subset S^{[n]}$ is the inclusion induced by the inclusion $C\subset S$ of a canonical curve, then by~\cite[equation~(33)]{OP},
\begin{gather*}
\big[S^{[n]}\big]^{\vir}=(-1)^{n}\iota_{*} \big[C^{[n]}\big].
\end{gather*}
Similarly, if $\Theta\in K(C)$ is a theta characteristic, then by~\cite[Theorem~15]{AJLOP},
\begin{gather*}
\O^{\vir}_{S^{[n]}}=(-1)^{n} \iota_{*} \det\Theta^{[n]}.
\end{gather*}
Virtual descendent integrals on $S^{[n]}$ can therefore be written in terms of tautological integrals on $C^{[n]}$.

\subsection{Universal series}\label{secuniv}
For fixed $\alpha_1,\dots, \alpha_l\in K(S)$, important structure emerges when the $k_i$ are allowed to vary and all series $Z_S(\alpha_1,\dots ,\alpha_l\mid k_1,\dots,k_l)$ are considered together.

Set
\begin{align}\label{exseries}
\hat{Z}_S( \alpha_1,\dots,\alpha_l)= \sum_{\substack{n\geq 0 \\ k_1,\dots,k_l\geq 0}} q^{n}(-m_1)^{k_1}\cdots (-m_l)^{k_l} \chi\big(S^{[n]}, \wedge^{k_1} \alpha^{[n]}_1 \otimes \dots \otimes \wedge^{k_l} \alpha_l^{[n]} \big).
\end{align}

Now, fix an $r$-tuple ${\bf r}=(r_1,\dots,r_l)$ of integers. By~\cite[Theorem~4.2]{EGL}, there exist {\em universal series}
\begin{gather}\label{univseries}
\mathsf{A}^{\bf{r}}, \mathsf{B}^{\bf{r}}, \mathsf{C}^{\bf{r}}_{i}, \mathsf{D}^{\bf{r}}_{i}, \mathsf{E}^{\bf{r}}_{i,j} \in \Q[q,m_1,\dots,m_l]
\end{gather}
for which, given any surface $S$ and any collection $\alpha_1,\dots,\alpha_l$ of classes in $K(S)$ such that $\mathrm{rank}(\alpha_i)=r_i$ on each component of $S$, one has
\begin{align}\label{univeq}
\hat{Z}_S( \alpha_1,\dots,\alpha_l)=(\mathsf{A}^{\bf{r}})^{\chi(\mathcal{O}_S)} (\mathsf{B}^{\bf{r}})^{\mathcal{K}_S^2}\prod_{i=1}^{l} (\mathsf{C}^{\bf{r}}_i)^{\mathcal{K}_S\cdot c_1(\alpha_i)}(\mathsf{D}^{\bf{r}}_{i})^{c_2(\alpha_i)} \prod_{1\leq i\leq j\leq l} (\mathsf{E}^{\bf{r}}_{i,j} )^{c_1(\alpha_i)\cdot c_1(\alpha_j)}.
\end{align}

More generally, such a factorization into universal series exists for generating series formed from integrals of multiplicative characteristic classes of tautological bundles and the tangent bundle. See~\cite[Theorem~4.2]{EGL} for a precise statement.

One such series of interest is the {\em Verlinde series}
\begin{gather*}
V_S(\alpha)= \sum_{n}q^{n}\chi\big(S^{[n]},\det\big(\alpha^{[n]}\big)\big).
\end{gather*}
The series $V_S(\alpha)$ is formed from a subset of terms of (\ref{exseries}) when $\alpha$ has positive rank. For $\alpha$ of negative rank, Serre duality implies a close relationship between $V_S(\alpha)$ and $V_S(-\alpha)$; see~\cite[Theorem~5.3]{EGL}. The series $V_S(\alpha)$ for $\alpha$ of rank $-1$, $0$ or $1$ are explicitly computed in~\cite[Theorem~5.3]{EGL}. A relationship between the universal series appearing in a factorization of $V_S(\alpha)$ and those appearing in a factorization of the {\em Segre series}
\begin{gather*}
\sum_n q^n\int_{S^{[n]}} s_{2n}\big(\alpha'^{[n]}\big)
\end{gather*}
for $\alpha'$ of rank $-\mathrm{rk}(\alpha)-1$ was proposed by Johnson in~\cite{J} and further explicated in~\cite[Conjecture 1]{MOP2}. This relationship is used in~\cite{MOP2, MOP3} to obtain conjectural formulas for $V_S(\alpha)$ for $\alpha$ of rank $-3$, $-2$, $2$ or $3$. It is expected that $V_S(\alpha)$ is an algebraic function of $q$ for any $\alpha$.

For general $\alpha_i$, not much is known or conjectured about $\hat{Z}_S(\alpha_1,\dots,\alpha_l)$ and their constituent universal series (\ref{univseries}); one framework using vertex operators is presented in~\cite{Z}. Our approach yields a new combinatorial expression for the series $\hat{Z}_S(\alpha_1,\dots,\alpha_l)$. However, it seems challenging to extract concisely stated consequences for the entire series, or even the Verlinde series. The individual coefficients in the $m$-variables $Z_S(\alpha_1,\dots,\alpha_l\mid k_1,\dots, k_l)$ studied in Theorem~\ref{descrat} seem more tractable from this perspective.

\subsection{Outline} In Section~\ref{C2togen}, we introduce an equivariant affine analog $\hat{Z}_{\C^2}$ of the series $\hat{Z}_S$ and demonstrate that Theorems~\ref{descrat} and~\ref{lbdescrat} follow from their equivariant analogs, Propositions~\ref{equivdesc} and~\ref{lbequivdesc}. In~Section~\ref{Macdonald}, we obtain Propositions~\ref{equivdesc} and~\ref{lbequivdesc} via the combinatorial identity Proposition~\ref{Macsym}. This identity is a slight modification of a formula obtained in~\cite{M} from a result of~\cite{GHT}. We also use this identity to prove Proposition~\ref{newdesc}.

\section{Descendent series from equivariant descendents}\label{C2togen}

\subsection{Equivariant descendents}
The following notation will be useful. For a variable or constant $m$ and a K-theory class $\alpha$ set \begin{gather*}
\wedge_{m}^{\bullet}\alpha=\sum_{k=0}^{\infty} (-m)^k \wedge^k \alpha.
\end{gather*}

Consider $\C^2$ equipped with the action of a torus $T=\mathrm{diag}(t_1,t_2)$ scaling the coordinate axes with weights $t_1^{-1}$ and $t_2^{-1}$. The action of $T$ on $\C^2$ lifts to an action on the Hilbert schemes $\big(\C^2\big)^{[n]}.$ The definition (\ref{tautdef}) is also valid in the equivariant setting; in this way, a class $\gamma\in K_T\big(\C^2\big)$ induces a tautological class $\gamma^{[n]} \in K_T\big(\big(\C^2\big)^{[n]}\big)$.

Definition (\ref{exseries}) can be extended to the equivariant setting. Given $\gamma_1,\dots,\gamma_l\in K_T\big(\C^2\big)$, define
\begin{gather*}
\hat{Z}_{\C^2}(\gamma_1,\dots,\gamma_l)(t_1,t_2)
\\ \qquad
{}=\sum_{n\geq 0} q^{n} \chi\big(\big(\C^2\big)^{[n]}, \wedge^{\bullet}_{m_1} \gamma^{[n]}_1 \otimes \dots \otimes \wedge^{\bullet}_{m_l} \gamma_l^{[n]} \big)\in \Q(t_1,t_2)[[q,m_1,\dots,m_l]].
\end{gather*}
Here, each term on the right-hand side is an equivariant Euler characteristic and can be regarded as a rational function on $T$.

\subsection{Localization on the Hilbert scheme}
We recall the following special case of K-theoretic equivariant localization. Let $M$ be a smooth complex variety equipped with an action of a complex torus $\mathbb{T}$ such that the fixed locus $M^{\mathbb{T}}$ is a nonempty finite set of points and let $\mathscr{F}\in K_{\mathbb{T}}(M)$.

\begin{Proposition}[{\cite[Theorem~3.5]{T}}]
There is an equality of $\mathbb{T}$-equivariant Euler characteristics
\begin{gather}\label{Keq}
\chi(M,\mathscr{F})=\sum_{p\in M^{\mathbb{T}}} \chi\bigg(p,\frac{\mathscr{F}|_{p}}{\wedge^{\bullet}_1 \mathcal{T}^{*}M|_p} \bigg) \in \Q(\mathbb{T}).
\end{gather}
\end{Proposition}

When applied to the Hilbert scheme of points, (\ref{Keq}) yields a combinatorial description of $\hat{Z}_{\C^2}(\gamma_1,\dots,\gamma_l).$ We record this description.

A {\em Young diagram} $\lambda$ is a finite subset of $\Z^2_{\geq 0}$ satisfying the following property: if $(c_1,c_2)\in \lambda$, then for any $(c_1',c_2')\in\Z^2_{\geq 0}$ such that $c'_1\leq c_1$ and $c'_2\leq c_2$, one also has $(c'_1,c'_2)\in \lambda$.

We can associate to a Young diagram $\lambda$ the point $p_{\lambda}\in \big(\C^2\big)^{[|\lambda|]}$ cut out by the monomial ideal
\begin{gather*}
\mathop{\rm Span}\bigl\{ x_1^{b_1}x_2^{b_2}\mid (b_1,b_2)\not\in \lambda\bigr\} \subset \C[x_1,x_2]=H^0(\mathcal{O}_{\C^2}).
\end{gather*}
The $T$-fixed locus of $\big(\C^2\big)^{[n]}$ consists of the points $p_{\lambda}$ with $\lambda$ of size $n$.

For $\gamma\in K_T\big(\C^2\big)$, let $\chi(\gamma|_{0})\in \Z\big[t_1^{\pm},t_2^{\pm}\big]$ denote the $T$-character of the fiber of $\gamma$ over the origin. For $\lambda$ of size $n$, the fiber $\gamma^{[n]}|_{p_{\lambda}}$ has $T$-character
\begin{gather*}
\chi(\gamma|_{0}) \sum_{(c_1,c_2)\in \lambda}t_1^{c_1}t_2^{c_2}.
\end{gather*}
In particular, if
\begin{gather*}
\chi(\gamma|_{0})=\sum_{i} v_i -\sum_{j} w_j,
\end{gather*}
where each $v_i$ and $w_j$ is a $T$-weight (a Laurent monomial in $t_1$ and $t_2$,) then the fiber $ \wedge^{\bullet}_m \gamma^{[n]}|_{p_{\lambda}}$ has $T$-character
\begin{gather*}
\Exp\biggl[{-}m\chi(\gamma|_{0}) \sum_{(c_1,c_2)\in\lambda}t_1^{c_1}t_2^{c_2}\biggr]=\prod_{(c_1,c_2)\in \lambda} \frac{\prod_{i}\big(1-mv_it_1^{c_1}t_2^{c_2}\big)}{\prod_{j}\big(1-mw_jt_1^{c_1}t_2^{c_2}\big)};
\end{gather*}
the definition of the plethystic exponential $\Exp$ is recalled in (\ref{expdef}).

Let $\lambda$ be a Young diagram. Given $(c_1,c_2)\in \lambda$, define the {\em leg length} $l((c_1,c_2))$ to be the largest integer $k$ for which $(c_1+k,c_2)\in \lambda$ and the {\em arm length} $a((c_1,c_2))$ to be the largest integer $k$ for which $(c_1,c_2+k)\in \lambda$.

For $\lambda$ of size $n$, the $T$-character of the fiber of the cotangent bundle $\mathcal{T}^{*}\big(\C^2\big)^{[n]}$ at $p_{\lambda}$ is computed in~\cite[Lemma~3.2]{ES} to be
\begin{gather*}
\sum_{\square\in\lambda} t_1^{l(\square)+1}t_2^{-a(\square)}+t_1^{-l(\square)}t_2^{a(\square)+1}.
\end{gather*}
Set $C_{\lambda}$ to be the $T$-character of ${\wedge^{\bullet}_1 \mathcal{T}^{*}\big(\C^2\big)^{[n]}|_{p_{\lambda}}};$ explicitly,
\begin{align}\label{Clambdadef} C_{\lambda}=\Exp\big[{-}\chi\big(p_{\lambda},\mathcal{T}^{*}\big(\C^2\big)^{[n]}|_{p_{\lambda}}\big)\big]= \prod_{\square\in\lambda}{\big(1- t_1^{l(\square)+1}t_2^{-a(\square)}\big) \big(1-t_1^{-l(\square)}t_2^{a(\square)+1}\big)}.
\end{align}

By (\ref{Keq}), we conclude that
\begin{align}\label{combZ}
\hat{Z}_{\C^2}(\gamma_1,\dots,\gamma_l)(t_1,t_2)=\sum_{\lambda}\frac{q^{|\lambda|}}{C_{\lambda}} \Exp\biggl[{-} \biggl(\sum_{j=1}^{l}m_j\cdot \chi(\gamma_j|_0)\biggr)\sum_{(c_1,c_2)\in \lambda} t_1^{c_1}t_2^{c_2} \biggr].
\end{align}

In mathematical physics, series of the form (\ref{combZ}) arise as rank $1$ Nekrasov partition functions of 5-dimensional supersymmetric gauge theories with fundamental matter.

\subsection[From C\textasciicircum{}2 to a general surface]{From $\mathbb{C}^2$ to a general surface}

\subsubsection{Arbitrary descendents}\label{arbitrary}

Fix $\gamma_1,\dots,\gamma_l\in K_T\big(\C^2\big)$. As ${\bf a}=(a_1,\dots,a_l)$ ranges over $\Z_{\geq 0}^{l}$, let
\begin{gather*}
g_{{\bf a}}(q)\in \Q(t_1,t_2)[[q]]
\end{gather*}
be the collection of series for which
\begin{align}\label{gdef} \hat{Z}_{\C^2}(\gamma_1,\dots,\gamma_l)(t_1,t_2)=\Exp\bigg[\frac{q}{(1-t_1)(1-t_2)}\bigg] \sum_{\bf{a}}m_1^{a_1}\cdots m_l^{a_l} g_{\bf{a}}(q).
\end{align}
We remark that
\begin{gather*}
\Exp\bigg[\frac{q}{(1-t_1)(1-t_2)}\bigg]=\hat{Z}_{\C^2}(\varnothing)(t_1,t_2),
\end{gather*}
which can be seen by Proposition~\ref{Macsym} or otherwise.

We formulate the following equivariant analog of Theorem~\ref{descrat}.

\begin{Proposition}\label{equivdesc}
The series $g_{{\bf a}}(q)$ are polynomials in $q$. Moreover,
\begin{gather*}
\deg_{q} g_{\bf{a}}\leq a_1+\dots+a_l.
\end{gather*}
\end{Proposition}
Proposition~\ref{equivdesc} is a consequence of (\ref{rpart}), whose proof is the subject of Section~\ref{ardescpf}.

Returning to our original problem, let $S$ be a projective surface and let $\alpha_1,\dots,\alpha_l$ be classes in $K(S)$. Without loss of generality, assume that $\alpha_j$ is of rank $r_j$ on all components of $S$.

Define $f_{{\bf a}}(q)\in \Q[[q]]$ to be the collection of series for which
 \begin{gather} \label{projser}
\hat{Z}_S(\alpha_1,\dots,\alpha_l)=\frac{1}{(1-q)^{\chi(\O_S)}}\sum_{{\bf a}} m_1^{a_1}\cdots m_l^{a_l} f_{{\bf a}}(q).\end{gather}

Proposition~\ref{equivdesc} implies the following rephrasing of Theorem~\ref{descrat}.

\begin{Corollary}\label{projdesc}
The series $f_{{\bf a}}(q)$ appearing in $(\ref{projser})$ are polynomials in $q$. Moreover,
\begin{gather*}
\deg_q f_{{\bf a}} \leq a_1+\dots+a_l.
\end{gather*}
\end{Corollary}

\begin{proof}
The first step is to use the argument of~\cite[Sections~4 and~5]{EGL} to reduce to the case when $S$ is toric and $\alpha_1,\dots,\alpha_l\in K(S)$ are torus-equivariant. By the factorization $(\ref{univeq})$, each coefficient of each $f_{\bf a}$ is a universal polynomial in the Chern numbers \begin{align}\label{cherndata}
\chi(\mathcal{O}_{S}),\ \mathcal{K}_{S}^2,\ \mathcal{K}_{S}\cdot c_1(\alpha_j),\ c_2(\alpha_j),\ c_1(\alpha_{j'})\cdot c_1(\alpha_{j''}),\qquad
1\leq j \leq l,\quad 1\leq j',j''\leq l.
\end{align}
Any polynomial that vanishes at all values of the form (\ref{cherndata}) when (for example) $S$ is toric and $\alpha_1,\dots,\alpha_l$ are torus-equivariant must be identically zero.

So, let $S$ be toric, let $T=\diag(t_1,t_2)$ act on $S$ with finitely many fixed points $s_i$ and let $\alpha_1,\dots,\alpha_l\in K_T(S)$. For each $s_i$, let $w_{i_1}$ and $w_{i_2}$ denote the cotangent weights at $s_i,$ let $U_i$ denote the toric chart centered at $s_i$, and set
\begin{gather*}
(\alpha_j)_i=\alpha_j|_{U_i}\in K_T(U_i)\cong K_T\big(\C^2\big).
\end{gather*}

The action of $T$ on $S$ lifts to an action on $S^{[n]}$. By (\ref{Keq}), there is the following equality of $T$-equivariant Euler characteristics
\begin{align*}
\sum_{n\geq 0} q^{n} \chi\big(S^{[n]}, \wedge^{\bullet}_{m_1} \alpha_1^{[n]} \otimes \dots \otimes \wedge^{\bullet}_{m_l} \alpha_l^{[n]} \big)=\prod_{i} \hat{Z}_{\C^2}((\alpha_1)_i,\dots,(\alpha_l)_i)(w_{i_1},w_{i_2}).
\end{align*}
So, the (nonequivariant) series $\hat{Z}_{S}$ can be recovered from $\hat{Z}_{\C^2}$ as follows: \begin{align}\label{toricfact}
\hat{Z}_{S}(\alpha_1,\dots,\alpha_l)=\bigg(\prod_{i} \hat{Z}_{\C^2}((\alpha_1)_i,\dots,(\alpha_l)_i)(w_{i_1},w_{i_2})) \bigg)\bigg|_{t_1=1,t_2=1}.
\end{align}

Moreover, by $T$-equivariant localization on $S$, one has
\begin{gather*}
\chi(\O_{S})=\bigg(\sum_{i} \frac{1}{(1-w_{i_1})(1-w_{i_2})}\bigg)\bigg |_{t_1=1,t_2=1},
\end{gather*}
so that
\begin{align}\nonumber
\bigg(\prod_{i} \Exp\bigg [\frac{q}{(1-w_{i_1})(1-w_{i_2})}\bigg ] \bigg)\bigg |_{t_1=1,t_2=1}&=\Exp\bigg[\bigg(\sum_{i} \frac{q}{(1-w_{i_1})(1-w_{i_2})}\bigg)\bigg |_{t_1=1,t_2=1}\bigg]\nonumber
\\
&=\Exp[q\chi(\O_S)]=\frac{1}{(1-q)^{\chi(\O_S)}}.
\label{prefeq}
\end{align}

By (\ref{prefeq}), the specialization at $t_1=1$, $t_2=1$ of the product of prefactors of each term on the right-hand side of (\ref{toricfact}) matches the denominator of the right-hand side of (\ref{projser}). Applying Proposition~\ref{equivdesc} to each factor in (\ref{toricfact}), we obtain the corollary. \end{proof}

\subsubsection[Descendents with alpha\_1 a line bundle]{Descendents with $\boldsymbol{\alpha_1}$ a line bundle}

We now turn to Theorem~\ref{lbdescrat}, which we will also deduce from an equivariant analog, Proposition~\ref{lbequivdesc}. The argument is a slightly more intricate version of that of Section~\ref{arbitrary}.

Fix $\gamma_1,\dots,\gamma_l\in K_T\big(\C^2\big)$ such that $\gamma_1$ is the class of a equivariant line bundle. Set
\begin{gather*}
u=\chi(\gamma_1|_0).
\end{gather*}
Note that $u$ is a monomial (with coefficient 1) in $t_1^{\pm}$ and $t_2^{\pm}$. As $\tilde{\mathbf{a}}=(a_2,\dots,a_l)$ ranges over $\Z_{\geq 0}^{l-1}$, let
\begin{gather*}
\tilde{g}_{\tilde{\bf{a}}}(q,m_1)\in \Q(t_1,t_2)[[q,m_1]]
\end{gather*}
be the series such that
\begin{align}\label{gtildedef}
\hat{Z}_{\C^2}(\gamma_1,\dots,\gamma_l)=\Exp\bigg[\frac{q-qm_1u}{(1-t_1)(1-t_2)}\bigg]
\sum_{\tilde{\bf{a}}} m_2^{a_2}\cdots m_l^{a_l}\tilde{g}_{\tilde{\bf{a}}}.
\end{align}

\begin{Proposition}\label{lbequivdesc}
The series $\tilde{g}_{\tilde{\bf a}}(q,m_1)$ are Laurent expansions of rational functions in $q$ and~$m_1$ of the form
\begin{gather*}
\frac{s\big(t_1^{\pm}, t_2^{\pm},q,m_1\big)}{\prod_{w}(1-w)\prod_{w'}(1-m_1w')},
\end{gather*}
where $s$ is a polynomial and $w$ and $w'$ range over finitely many $T$-weights.
\end{Proposition}

Proposition~\ref{lbequivdesc} is a consequence of (\ref{lbrpart}), whose proof is the subject of Section~\ref{lbdescpf}.

Now, let $S$ be a projective surface and fix $\alpha_1,\dots,\alpha_l\in K(S)$ such that $\alpha_1$ is the class of a~line bundle. Again, assume that $\alpha_j$ is of rank $r_j$ on all components of $S$.

Let $\tilde{f}_{\tilde{\bf{a}}}(q,m_1)\in \Q[[q,m_1]]$ be the collection of series for which
\begin{gather*}
\hat{Z}_S(\alpha_1,\dots,\alpha_l)=\frac{(1-qm_1)^{\chi(\alpha_1)}} {(1-q)^{\chi(\mathcal{O}_S)}}\sum_{\tilde{\bf{a}}} m_2^{a_2}\cdots m_l^{a_l}\tilde{f}_{\tilde{\bf{a}}}(q,m_1).
\end{gather*}

\begin{Corollary}
The series $\tilde{f}_{\tilde{\bf a}}(q)$ is the expansion in $q$ and $m_1$ of a rational function whose denominator is a power of $(1-qm_1)$.
\end{Corollary}

\begin{proof}
Again, it suffices to prove the corollary for toric $S$ and torus equivariant $\alpha_1,\dots,\alpha_l\in K(S)$. Let $s_i$, $w_{i_1}$, $w_{i_2}$, $U_i$ and $(\alpha_j)|_i$ be as in the proof of Corollary~\ref{projdesc}. By equivariant localization on $S$, we have \begin{align}\label{lbprefeq}
\nonumber\bigg(\prod_{i} \Exp\bigg [\frac{q-qm_1(\alpha_1)|_i}{(1-w_{i_1})(1-w_{i_2})}\bigg ] \bigg)\bigg |_{t_1=1,t_2=1}&=\Exp\bigg[\bigg(\sum_i \frac{q-qm_1(\alpha_1)|_i}{(1-w_{i_1})(1-w_{i_2})}\bigg)\bigg |_{t_1=1,t_2=1}\bigg ]
\\
&\nonumber=\Exp[q\chi(\mathcal{O}_S)-qm_1\chi(\alpha_1)]
\\
&=\frac{(1-qm_1)^{\chi(\alpha_1)}}{(1-q)^{\chi(\mathcal{O}_S)}}.
\end{align}

Now, plug equation (\ref{gtildedef}) into (\ref{toricfact}) and use (\ref{lbprefeq}) to combine prefactors. Putting the remaining expressions over a common denominator as needed, given some $\tilde{\bf{a}}$ we may write
\begin{gather*}
\tilde{f}_{\tilde{\bf{a}}}=\frac{r\big(t_1^{\pm}, t_2^{\pm},q,m_1\big)}{\prod_v (1-v)\prod_{v'} (1-qm_1v')}\bigg|_{t_1=1,t_2=1},
\end{gather*}
where $r$ is a polynomial and $v$ and $v'$ range over finitely many $T$-weights. In particular, if the rational function
\begin{align}\label{rseries}
\frac{r\big(t_1^{\pm}, t_2^{\pm},q,m_1\big)}{\prod_v (1-v)\prod_{v'} (1-qm_1v')}
\end{align}
is expanded in positive powers of $m_1$ and $q$, then each $q^{k}m_1^{k'}$-coefficient of the resulting series is well defined under the specialization $t_1=1$, $t_2=1$. We conclude that each $q^{k}m_1^{k'}$-coefficient of (\ref{rseries}) has no poles of the form $1-v$. By induction on the degree $k+k'$, it follows that each $q^{k}m_1^{k'}$-coefficient of the quotient
\begin{gather*}
\frac{r\big(t_1^{\pm}, t_2^{\pm},q,m_1\big)}{\prod_v (1-v)}
\end{gather*}
also has no poles of the form $1-v$ and is therefore a Laurent polynomial in $t_1$ and $t_2$. It follows that
\begin{align*}
\tilde{f}_{\tilde{\bf{a}}}&=\frac{r\big(t_1^{\pm}, t_2^{\pm},q,m_1\big)}{\prod_v (1-v)}\Big|_{t_1=1,t_2=1}\cdot \frac{1}{\prod_{v'} (1-qm_1v')}\Big|_{t_1=1,t_2=1}
\\
&=\frac{r\big(t_1^{\pm}, t_2^{\pm},q,m_1\big)}{\prod_v (1-v)}\Big|_{t_1=1,t_2=1}\cdot \frac{1}{\prod_{v'} (1-qm_1)}.\tag*{\qed}
\end{align*}
\renewcommand{\qed}{}
\end{proof}

\section{A Macdonald identity}\label{equivproof}

\subsection{Plethystic notation}\label{Macdonald}

Proposition~\ref{equivdesc} follows from a slight modification of a Macdonald polynomial identity obtained in~\cite[Section~7]{M}, where the identity is applied to find symmetries among conjectural expressions for mixed Hodge polynomials of certain character varieties. We recall this identity following the presentation in~\cite{M}.

Let $p_n$ denote the $n$-th power sum and let $\Sym$ denote the completion of the ring
\begin{gather*}
\Q(t_1,t_2)[p_1,p_2,\ldots]
\end{gather*}
of symmetric functions over $\Q(t_1,t_2)$ with respect to degree. We use the following ``plethystic notation''. Let
\begin{gather*}
X=\sum_{{\bf k}} c_{\bf k} {\bf x}^{\bf k}
\end{gather*}
denote a Laurent series where each $c_{\bf k}\in \Q$ and each ${\bf x}^{\bf k}$ is a Laurent monomial with coefficient~1 in $t_1$, $t_2$ and the additional variables $q, m_1,m_2,\dots$. Then, we set
\begin{gather*}
p_n[X]=\sum_{{\bf k}} c_{\bf k} \big({\bf x}^{\bf k}\big)^n.
\end{gather*}
For arbitrary $F\in \Sym$, the value of the expression $F[X]$ is defined by stipulating that the assignment $F\mapsto F[X]$ is a ring homomorphism. The plethystic exponential $\Exp$ is defined as \begin{align}\label{expdef}
\Exp[X]=\exp\bigg(\sum_{n=0}^{\infty} \frac{p_n}{n}[X]\bigg)=1+p_1[X]+\frac{p_2+p_1^2}{2}[X]+\cdots.
\end{align}
Note that
\begin{gather*}
\Exp[X+Y]=\Exp[X]\cdot \Exp[Y].
\end{gather*}
In particular, if $x_i$ and $y_j$ are Laurent monomials with coefficient $1$, then
\begin{gather*}
\Exp\bigg[\sum_{i} x_i-\sum_{j} y_j \bigg]= \frac{\prod_{j}(1-y_j)}{\prod_{i}(1-x_i)},
\end{gather*}
where infinite products are taken in a suitable completion.

\subsection{Macdonald polynomials}

For a Young diagram $\lambda$, let $H_{\lambda}\in \Sym$ denote the corresponding Macdonald polynomial as normalized, for example, in~\cite[equation~(11)]{GH}.
Equivalent definitions may be found in~\cite[Theorem~2.2]{GH} and~\cite[Definition 4.1]{M}; note that the Macdonald polynomials are denoted $\tilde{H}_{\lambda}$ in~\cite{GH,GHT}.

We summarize the relevant properties of the symmetric polynomials $H_{\lambda}$.
\begin{itemize}\itemsep=0pt
\item The polynomial $H_{\lambda}$ is homogeneous of degree $|\lambda|$.
\item The polynomials $H_{\lambda}$ form a $\Q(t_1,t_2)$-basis of the space of symmetric functions.
\item If $x$ is a Laurent monomial with coefficient 1, then by~\cite[Corollary~2.1]{GH}, \begin{align}\label{specialvalues2}
 H_{\lambda}[1-x]&=\Exp\biggl[{-}x \sum_{(c_1,c_2)\in \lambda} t_1^{c_1}t_2^{c_2}\biggr]= \prod_{(c_1,c_2)\in \lambda} \big(1-xt_1^{c_1}t_2^{c_2}\big).
 \end{align}
In particular,
\begin{gather} \label{specialvalues1}
H_{\lambda}[1]=1,\qquad
H_{\lambda}[-1]=(-1)^{|\lambda|}\prod_{(c_1,c_2)\in \lambda} t_1^{c_1}t_2^{c_2}.
\end{gather}

\item The ring $\Sym$ carries the Macdonald scalar product $\langle\,,\rangle_{*}$ satisfying
\begin{align*}\nonumber
\langle H_{\lambda}, H_{\mu}\rangle_* = \delta_{\lambda,\mu}\cdot H_{\lambda}[-1]\cdot C_{\lambda},
\end{align*}
where
\begin{gather*}
C_{\lambda}= \prod_{\square\in\lambda}{\big(1- t_1^{l(\square)+1}t_2^{-a(\square)}\big)\big(1-t_1^{-l(\square)}t_2^{a(\square)+1}\big)}
\end{gather*}
denotes the Laurent polynomial (\ref{Clambdadef}).
In particular, the scalar product $\langle\,,\rangle_{*}$ has reproducing kernel \begin{gather}
\label{cauchy}
\sum_{\lambda}\frac{H_{\lambda}[X] H_{\lambda}[Y]} {H_{\lambda}[-1] \cdot C_{\lambda}} .
\end{gather}
\end{itemize}

\subsection{A plethystic symmetry}
We will deduce Proposition~\ref{equivdesc} from a special case of the following symmetry.

\begin{Proposition}[{cf.~\cite[Section~7]{M}}]\label{Macsym}
For Laurent series $X$ and $Y$, the expression
\begin{align}\label{Macid}
\Exp\bigg[\frac{Y}{(1-t_1)(1-t_2)}\bigg]\sum_{\lambda}\frac{H_{\lambda}[X]}{C_{\lambda}} \Exp\biggl[{-}Y \sum_{(c_1,c_2)\in \lambda} t_1^{c_1}t_2^{c_2}\biggr]
\end{align}
is symmetric under exchange of $X$ and $Y$.
\end{Proposition}

\begin{proof} We slightly modify the proof from~\cite{M}, which uses the Macdonald polynomial identity~\cite[Theorem~I.3]{GHT}.
Denote by $U$, $U^*$ and $\nabla\colon \Sym\to \Sym$ the operators
\begin{gather*}
(UF)[X]=F[1+X], \qquad
(U^*F)[X]=\Exp \biggl[{-}\frac{X}{(1 - t_1)(1 - t_2)}\biggr] F[X], \\
\nabla H_{\lambda} =H_{\lambda}[-1]\cdot H_{\lambda}.
\end{gather*}

Then,~\cite[Theorem~I.3]{GHT} states that
\begin{align}\label{Videntity}
(\nabla U^*U)H_{\lambda}[X]=\Exp\bigg[\frac{X}{(1-t_1)(1-t_2)}\bigg] \Exp\biggl[{-}X \sum_{(c_1,c_2)\in \lambda} t_1^{c_1}t_2^{c_2}\biggr].
\end{align}
In other words, the composition $\nabla U^*U$ sends Macdonald polynomials to normalized $T$-weights of tautological bundles at fixed points on the Hilbert scheme.

By~\cite[Proposition 1.11b]{GHT}, the operators $U$ and $U^*$ are adjoint with respect to $\langle\,,\rangle_{*}$. Moreover, the operator $\nabla$ is self-adjoint. So, the operator $\nabla U^*U\nabla$ is self-adjoint. As (\ref{cauchy}) is the reproducing kernel for $\langle\,,\rangle_*$, we have
\begin{align}\label{preswap}
(\nabla U^* U\nabla)_{Y}\bigg[\sum_{\lambda} {\frac{H_{\lambda}[X] H_{\lambda}[Y]} {H_{\lambda}[-1]C_{\lambda}} }\bigg]=(\nabla U^*U\nabla)_{X}\bigg[\sum_{\lambda} {\frac{H_{\lambda}[X] H_{\lambda}[Y]} {H_{\lambda}[-1]C_{\lambda}} }\bigg],
\end{align}
where the subscript $X$ or $Y$ denotes action on symmetric functions taking that argument.

By (\ref{Videntity}), the left-hand side of (\ref{preswap}) equals
\begin{align*}
\Exp\bigg[\frac{Y}{(1-t_1)(1-t_2)}\bigg]\sum_{\lambda}\frac{H_{\lambda}[X]}{C_{\lambda}}\Exp\bigg[{-}Y \sum_{(c_1,c_2)\in \lambda} t_1^{c_1}t_2^{c_2}\bigg].
\end{align*}
As (\ref{cauchy}) is symmetric under exchange of $X$ and $Y$, the proposition follows.
\end{proof}

It would be interesting to have a geometric interpretation of Proposition~\ref{Macsym}.

\subsection{Specialization}
To prove Propositions~\ref{equivdesc} and~\ref{lbequivdesc}, we apply Proposition~\ref{Macsym} to control the series $\hat{Z}_{\C^2}(\gamma_1,\dots,\gamma_l)$. With (\ref{combZ}) in mind, for $j=1,\dots,l$, we set
\begin{gather*}
u_j=\chi(\gamma_j|_0)\in \Z\big[t_1^{\pm},t_2^{\pm}\big]
\end{gather*}
to be the $T$-character of the fiber of $\gamma_j$ at the origin $0\in\C^2$.

\subsubsection{Arbitrary descendents}\label{ardescpf}
Apply the specialization
\begin{gather*}
X=q, \qquad Y=\sum_{j=1}^{l} m_ju_j
\end{gather*}
to expression (\ref{Macid}).
By (\ref{specialvalues1}),~(\ref{specialvalues2}) and (\ref{combZ}), this specialization equals
\begin{gather}
\Exp\bigg[\frac{\sum_{j=1}^{l} m_ju_j}{(1-t_1)(1-t_2)}\bigg] \sum_{\lambda} \frac{q^{|\lambda|}}{C_{\lambda}}\Exp\Biggl[{-}\bigg(\sum_{j=1}^{l} m_ju_j\bigg) \sum_{(c_1,c_2)\in \lambda}t_1^{c_1}t_2^{c_2} \Biggr]\nonumber
\\ \qquad
{}=\Exp\bigg[\frac{\sum_{j=1}^{l} m_ju_j}{(1-t_1)(1-t_2)}\bigg] \hat{Z}_{\C^2}(\gamma_1,\dots,\gamma_l)(t_1,t_2).\label{specl}
\end{gather}
Applying Proposition~\ref{Macsym}, we find that the series (\ref{specl}) equals
\begin{align*}
\Exp\bigg[\frac{q}{(1-t_1)(1-t_2)}\bigg] \sum_{\lambda} \frac{H_{\lambda} \big(\sum_{j=1}^{l} m_ju_j\big)}{C_{\lambda}} { \prod_{(c_1,c_2)\in \lambda} \big(1-qt_1^{c_1}t_2^{c_2}\big)}.
\end{align*}

Recall the definition of $g_{\bf{a}}$ from (\ref{gdef}). We conclude that
\begin{align}\label{rpart} \nonumber
\sum_{\bf{a}} &m_1^{a_1}\cdots m_l^{a_l} g_{\bf{a}}(q)\\&= \Exp\bigg[{-}\frac{\sum_{j=1}^{l} m_ju_j}{(1-t_1)(1-t_2)}\bigg] \sum_{\lambda} \frac{H_{\lambda} \big(\sum_{j=1}^{l} m_ju_j\big)}{C_{\lambda}} { \prod_{(c_1,c_2)\in \lambda} \big(1-qt_1^{c_1}t_2^{c_2}\big)}.
\end{align}

As $H_{\lambda}$ is homogeneous of degree $|\lambda|$, only partitions $\lambda$ of size at most $a_1+\dots+a_l$ can contribute $m_1^{a_1}\cdots m_l^{a_l}$-terms to the right-hand side of (\ref{rpart}). As the largest power of $q$ that can appear in the $\lambda$-summand of (\ref{rpart}) is $|\lambda|$, Proposition~\ref{equivdesc} follows.

\subsubsection{Computation of descendent series}\label{compdesc}

For any fixed ${\bf{a}}$, equation (\ref{rpart}) yields an expression for $g_{\bf{a}}$ as a finite sum in terms of Macdonald polynomials. This expression can be used to obtain new formulas for descendent series.

For example, we prove Proposition~\ref{newdesc}. Let $\gamma\in K_T\big(\C^2\big)$ be a rank $3$ vector bundle. We may write
\begin{gather*}
\chi(\gamma|_{0})= v_1+v_2+v_3,
\end{gather*}
where $v_1$, $v_2$ and $v_3$ are monomials with coefficient 1 in $t_1^{\pm}$ and $t_2^{\pm}$. For this section, set $m=m_1$. Equation (\ref{rpart}) implies
\begin{align}\label{threeCdesc} Z_{\C^2}(\gamma)=\Exp\left[\frac{q-m(v_1+v_2+v_3)}{(1-t_1)(1-t_2)}\right] \sum_{\lambda}\frac{H_{\lambda}(m(v_1+v_2+v_3))}{C_{\lambda}}\prod_{(c_1,c_2)\in \lambda}\big(1-qt_1^{c_1}t_2^{c_2}\big).\!\!\!
\end{align}

Set
\begin{gather*}
v^{(1)} = \chi(\gamma|_0)=v_1 + v_2 + v_3, \\
v^{(2)}=\chi\big({\wedge}^2 \gamma|_0\big)=v_1v_2+v_1v_3 + v_2v_3,\\
v^{(3)} =\chi\big({\wedge}^3 \gamma|_0\big)=v_1v_2v_3.
\end{gather*}
We explicitly compute the right-hand side of (\ref{threeCdesc}) to order $3$ in $m$ using formulas for $H_{\lambda}$ as listed in~\cite[Section~7.12]{OS}, for example. After writing the result in terms of the plethystic exponential, we obtain
\begin{align}
Z_{\C^2}(\gamma)\!=\!{}&\Exp\bigg[\frac{q\!-\!mqv^{(1)}\!\!+\!m^2(q\!-\!q^2)v^{(2)}\! +\!m^3(q\!-\!q^2)\big(qv^{(2)}v^{(1)}\! \!+\!q(t_1\!+\!t_2)v^{(3)}\!\!-\!(1\!+\!q)v^{(3)}\big) }{(1\!-\!t_1)(1\!-\!t_2)} \bigg] \nonumber
\\ \qquad
&+O\big(m^4\big).\label{descexpol}
\end{align}
We rewrite (\ref{descexpol}) as
\begin{align}
Z_{\C^2}(\gamma)=\Exp&\big[ q\chi\big(\mathcal{O}_{\C^2}\big)-mq\chi(\gamma)+m^2\big(q-q^2\big)\chi\big({\wedge}^2\gamma\big)
+m^3\big(q-q^2\big)\big(q\chi\big({\wedge}^2\gamma\otimes \gamma\big)\nonumber
\\
&+q\chi\big(\mathcal{T}^*\C^2\otimes\wedge^3\gamma\big)-(1+q)\chi \big({\wedge}^3\gamma\big)\big)\big]+ O\big(m^4\big).
\label{descexgeo}
\end{align}

Now, let $S$ be a toric projective surface and $\mathcal{V}$ be a torus-equivariant rank 3 vector bundle on~$S$. As explained in Section~\ref{arbitrary}, it suffices to prove Proposition~\ref{newdesc} for such $S$ and $\mathcal{V}$. Equations (\ref{descexgeo}), and~(\ref{toricfact}) and equivariant localization imply that
\begin{align}
Z_{S}(\mathcal{V})&=\Exp\big[ q\chi(\mathcal{O}_S)-mq\chi(\mathcal{V})+m^2\big(q-q^2\big)\chi\big({\wedge}^2\mathcal{V}\big)+m^3\big(q-q^2\big) \big(q\chi\big({\wedge}^2\mathcal{V}\otimes \mathcal{V}\big)\nonumber
\\
&\hphantom{=\Exp\big[}+q\chi\big(\mathcal{T}^*S\otimes\wedge^3\mathcal{V}\big) -(1+q)\chi\big({\wedge}^3\mathcal{V}\big)\big)\big]+ O\big(m^4\big) \nonumber
\\
&=\frac{(1-mq)^{\chi(\mathcal{V})}}{(1-q)^{\chi(\mathcal{O}_S )}}\bigg(\frac{1-m^2q^2}{1-m^2q}\bigg)^{\chi(\wedge^2\mathcal{V})}\nonumber
\bigg(\frac{1-m^3q^3}{1-m^3q^2}\bigg)^{\chi(\wedge^2\mathcal{V}\otimes\mathcal{V})}
\\
&\phantom{=}\cdot\bigg(\frac{1-m^3q^3}{1-m^3q^2}\bigg)^{\chi(\mathcal{T}^*S\otimes \wedge^3\mathcal{V})}\bigg(\frac{1-m^3q}{1-m^3q^3}\bigg)^{\chi(\wedge^3\mathcal{V})}+O\big(m^4\big) .
\label{projdescexpgeo}
\end{align}
Extracting the $m^3$ term of both sides of (\ref{projdescexpgeo}) yields Proposition~\ref{newdesc}.

\subsubsection[Descendents with gamma\_1 a line bundle]{Descendents with $\boldsymbol{\gamma_1}$ a line bundle}\label{lbdescpf}

When $\gamma_1$ is the class of a line bundle, Proposition~\ref{Macsym} yields extra information about the structure of $\hat{Z}_{\C^2}(\gamma_1,\dots,\gamma_l)$.

By the homogeneity of $H_{\lambda}$ and (\ref{specialvalues2}), we have
\begin{gather*}
H_{\lambda}[q-qm_1]=q^{|\lambda|}H_{\lambda}[1-m_1]=q^{|\lambda|}\,
\Exp\biggl[-m_1\sum_{(c_1,c_2)\in\lambda}t_1^{c_1}t_2^{c_2}\biggr].
\end{gather*}
As $\gamma_1$ is a line bundle, the $T$-character $u_1=\chi(\gamma_1|_0)$ is a monomial (with coefficient 1) in $t_1^{\pm}$ and~$t_2^{\pm}$. So, applying the specialization
\begin{gather*}
X=q-qm_1u_1, \qquad Y=\sum_{j=2}^{l} m_ju_j
\end{gather*}
to expression (\ref{Macid}), we obtain
{\samepage\begin{gather}\nonumber
\Exp\Bigg[\frac{\sum_{j=2}^{l} m_ju_j}{(1-t_1)(1-t_2)}\Bigg]\sum_{\lambda} \frac{q^{|\lambda|}}{C_{\lambda}}\,\Exp\Biggl[-\bigg(\sum_{j=1}^{l} m_ju_j\bigg)\, \sum_{(c_1,c_2)\in \lambda} t_1^{c_1}t_2^{c_2} \Biggr]
\\ \qquad
{} = \Exp\Bigg[\frac{\sum_{j=2}^{l} m_ju_j}{(1-t_1)(1-t_2)}\Bigg]\hat{Z}_{\C^2}(\gamma_1,\dots,\gamma_l)(t_1,t_2) .\label{lbspecl}
\end{gather}}%
By Proposition~\ref{Macsym}, the series (\ref{lbspecl}) equals
\begin{align*}
\Exp\bigg[\frac{q-qm_1u_1}{(1-t_1)(1-t_2)}\bigg] \sum_{\lambda} \frac{H_{\lambda} \big(\sum_{j=2}^{l} m_ju_j\big)}{C_{\lambda}} \prod_{(c_1,c_2)\in \lambda} \frac{1-qt_1^{c_1}t_2^{c_2}}{1-qm_1u_1t_1^{c_1}t_2^{c_2}}.
\end{align*}

Recall the definition of $\tilde{g}_{\tilde{\bf{a}}}$ from (\ref{gtildedef}). We conclude that
\begin{align}\nonumber
\sum_{\bf{\tilde{a}}} m_2^{a_2}\cdots m_l^{a_l} \tilde{g}_{\tilde{\bf{a}}}(q,m_1)
={}&\Exp\Biggl[-\frac{\sum_{j=2}^{l} m_ju_j}{(1-t_1)(1-t_2)}\Biggr]
\\
&{}\cdot \sum_{\lambda} \frac{H_{\lambda} \big(\sum_{j=2}^{l} m_ju_j\big)}{C_{\lambda}} \prod_{(c_1,c_2)\in \lambda} \frac{1-qt_1^{c_1}t_2^{c_2}}{1-qm_1u_1t_1^{c_1}t_2^{c_2}}.
\label{lbrpart}
\end{align}
Only partitions of size at most $a_2+\dots+a_l$ can contribute $m_2^{a_2}\cdots m_l^{a_l}$-terms to the left-hand side of (\ref{lbrpart}). Proposition~\ref{lbequivdesc} follows.

\subsection*{Acknowledgements}

I thank Lothar G\"{o}ttsche, Anton Mellit and Richard Thomas for feedback and related conversations, as well as Drew Johnson, Woonam Lim, Dragos Oprea and Rahul Pandharipande for discussions and correspondence regarding Hilbert and Quot schemes. In particular, I thank Dragos Oprea for his suggestion to study the series~(\ref{lbspec}). I also thank the anonymous referees for their feedback. This work was supported by the EPSRC through grant EP/R013349/1 and the NSF through grant DMS-1902717.


\pdfbookmark[1]{References}{ref}
\LastPageEnding

\end{document}